\documentclass{article}
\usepackage[nonatbib,preprint]{neurips_2021}
\usepackage[numbers]{natbib}
\usepackage[utf8]{inputenc} 
\usepackage[T1]{fontenc}    
\usepackage{hyperref}       
\usepackage{url}            
\usepackage{booktabs}       
\usepackage{amsfonts}       
\usepackage{nicefrac}       
\usepackage{microtype}      

\usepackage{amsmath,mathrsfs,amssymb,mathtools,bm}
\usepackage{scalerel}
\usepackage{amsthm}
\usepackage{txfonts}
\usepackage{fontawesome}
\usepackage{wrapfig}
\usepackage{todonotes}
\usepackage{xcolor}
\usepackage{setspace}
\usepackage{bm,bbm}
\usepackage{paralist}
\usepackage{enumitem}

\usepackage[ruled,linesnumbered,noresetcount,vlined]{algorithm2e}
\makeatletter

\newcommand{\nosemic}{\renewcommand{\@endalgocfline}{\relax}}
\newcommand{\dosemic}{\renewcommand{\@endalgocfline}{\algocf@endline}}
\let\oldnl\nl
\newcommand{\nonl}{\renewcommand{\nl}{\let\nl\oldnl}}

\SetKwFor{Repeat}{repeat}{:}{endw}
\makeatother

\makeatletter
\newcommand{\oset}[3][0ex]{%
  \mathrel{\mathop{#3}\limits^{
    \vbox to#1{\kern-2\ex@
    \hbox{$\scriptstyle#2$}\vss}}}}
\makeatother
\newcommand{\optimal}[1]{\oset{\scalebox{.5}{$\star$}}{#1}\!}

\usepackage{letltxmacro}
\LetLtxMacro\orgvdots\vdots
\LetLtxMacro\orgddots\ddots

\newtheorem{problem*}{Problem}

\newtheorem{theorem}{Theorem}

\newtheorem*{example*}{Example}

\DeclareMathOperator*{\argmax}{argmax}
\DeclareMathOperator*{\argmin}{argmin}
\DeclareMathOperator*{\minimize}{minimize}

\usepackage{multirow,mdframed}
\mdfsetup{skipabove=0pt,skipbelow=0pt}
\mdfdefinestyle{model}{%
  innertopmargin=0pt,
  innerbottommargin=0pt,
  innerleftmargin=0pt,
  innerrightmargin=0pt,
  skipbelow=0pt,%
  skipabove=0pt,
  splitbottomskip=0pt,
  splittopskip=0pt,
  leftmargin =0pt,%
    rightmargin=0pt,
  splittopskip=0pt,
    usetwoside=false,
}
\usepackage{float}
\floatstyle{ruled}
\newfloat{model}{thp}{lop}
\floatname{model}{Model}

\newcommand*{\defeq}{\stackrel{\text{def}}{=}}

\newcommand{\cC}{\mathcal{C}}

\newcommand{\cI}{\mathcal{I}} 
 \newcommand{\cN}{\mathcal{N}}
\newcommand{\cO}{\mathcal{O}}

 \newcommand{\cY}{\mathcal{Y}}
 \newcommand{\cX}{\mathcal{X}}

 \newcommand{\RR}{\mathbb{R}}

\newcommand{\st}{\text{subject to}}

\usepackage{esvect}

\newcommand\hugP[1]{\left(#1\right)}

\title{Learning Hard Optimization Problems:\\A Data Generation Perspective}

\author{James Kotary\\
  Syracuse University\\
  \texttt{jkotary@syr.edu} \\
  \And
  Ferdinando Fioretto\\
  Syracuse University\\
  \texttt{ffiorett@syr.edu} \\
  \And
  Pascal Van Hentenryck\\
  Georgia Institute of Technology\\
  \texttt{pvh@isye.gatech.edu} \\
}

\begin{document}

\maketitle\sloppy\allowdisplaybreaks

\begin{abstract}
Optimization problems are ubiquitous in our societies and are present 
in almost every segment of the economy. Most of these optimization 
problems are NP-hard and computationally demanding, often requiring 
approximate solutions for large-scale instances. Machine learning
frameworks that learn to approximate solutions to such hard
optimization problems are a potentially promising avenue to address 
these difficulties, particularly when many closely related problem 
instances must be solved repeatedly. Supervised learning frameworks 
can train a model using the outputs of pre-solved
instances. However, when the outputs are themselves approximations, 
when the optimization problem has symmetric solutions, and/or when 
the solver uses randomization, solutions to closely related instances 
may  exhibit large differences and the learning task can become 
inherently more difficult. This paper demonstrates this critical challenge, connects 
the volatility of the training data to the ability of a model to 
approximate it, and proposes a method for producing (exact or 
approximate) solutions to optimization problems that are more amenable 
to supervised learning tasks.  The effectiveness of the method is 
tested on hard non-linear nonconvex and discrete combinatorial problems.
\end{abstract}

\section{Introduction} 
\label{sec:introduction}

Constrained optimization (CO) is in daily use in our society, with
applications ranging from supply chains and logistics, to electricity
grids, organ exchanges, marketing campaigns, and manufacturing to
name only a few. 
Two classes of hard optimization problems of particular interest in many
fields are (1) \emph{combinatorial optimization problems} and
(2) \emph{nonlinear constrained problems}. Combinatorial optimization
problems are characterized by discrete search spaces and have
solutions that are combinatorial in nature, involving for instance,
the selection of subsets or permutations, and the sequencing or
scheduling of tasks. Nonlinear constrained problems may have
continuous search spaces but are often characterized by highly
nonlinear constraints, such as those arising in electrical power
systems whose applications must capture physical laws such as Ohm’s
law and Kirchhoff’s law in addition to  engineering operational
constraints. Such CO problems are often NP-Hard and may be
computationally challenging in practice, especially for large-scale
instances.  

While the AI and Operational Research communities have contributed
fundamental advances in optimization in the last decades, the
complexity of some problems often prevents them from being adopted in
contexts where many instances must be solved over a long-term
horizon (e.g., multi-year planning studies) or when solutions must be
produced under time constraints. Fortunately, in many practical
cases, including the scheduling and energy problems motivating this
work, one is interested in solving many problem instances sharing
similar patterns. Therefore, the application of deep learning methods
to aid the solving of computationally challenging constrained
optimization problems appears to be a natural approach and has gained
traction in the nascent area at the intersection between CO and
ML \cite{bengio2020machine,kotary2021end,vesselinova2020learning}.
In particular, supervised learning frameworks can train a model using 
pre-solved CO instances and their solutions. However, learning the 
underlying combinatorial structure of the problem or learning approximations of optimization problems with hard physical and engineering constraints may be an extremely difficult task. While much of the recent research at 
the intersection of CO and ML has focused on learning good CO 
approximations in jointly training prediction and optimization models 
\cite{balcan2018learning,khalil2016learning,nair2020solving,nowak2018revised,vinyals2017pointer} 
and incorporating optimization algorithms into differentiable 
systems \cite{amos2019optnet,vlastelica2020differentiation,wilder2018melding,mandi2019smart}, 
learning the combinatorial structure of CO problems remains an elusive task.

Beside the difficulty of handling hard constraints, which will almost
always exhibit some violations, two interesting challenges have
emerged: the presence of multiple, often symmetric, solutions, and
the learning of approximate solution methods. The first challenge
recognizes that an optimization problem may not have a
optimal solution.
This challenge is illustrated in Figure \ref{fig:motivation_intro}, 
\begin{wrapfigure}[12]{r}{150pt}
  \vspace{-12pt}
  \centering
  \includegraphics[width=0.8\linewidth]{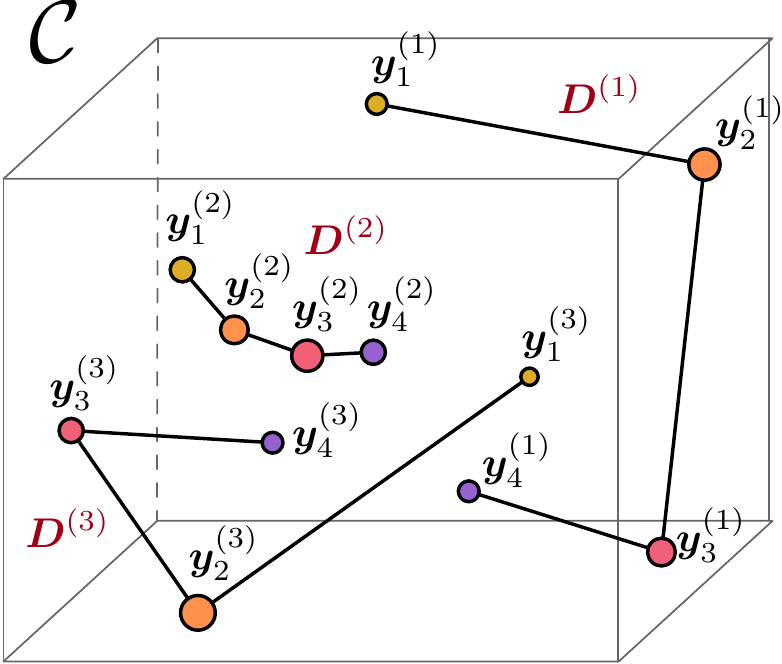}
  \caption{\small Co-optimal datasets due to symmetries. \label{fig:motivation_intro}}
\end{wrapfigure}
where the various $\bm{y}^{(i)}$ represent \emph{optimal} solutions 
to CO instances $\bm{x}^{(i)}$ and $\cC$ the feasibility space.
As a result,  a combinatorial number of possible datasets may be generated. While equally valid as optimal solutions, some sets follow patterns which are more meaningful and recognizable. Symmetry breaking is of course a major area of combinatorial optimization and may alleviate some of these issues. But different instances may not break symmetries in the same fashion, thus creating datasets that are harder to learn.

The second challenge comes from realities in the application domain.
Because of time constraints, the solution technique may return a
sub-optimal solution. Moreover, modern combinatorial optimization
techniques often use randomization and large neighborhood search to
produce high-quality solutions quickly. Although these are widely
successful, different runs for the same, or similar, instances may
produce radically different solutions. As a result, learning the
solutions returned by these approximations may be inherently more
difficult. These effects may be viewed as a source of noise that
obscures the relationships between training data and their target
outputs. Although this does not raise issues for optimization
systems, it creates challenging learning tasks.

This paper demonstrates these relations, connects the volatility of
the training data to the ability of a model to approximate it, and
proposes a method for producing (exact or approximate) solutions to
optimization problems that are more amenable to supervised learning
tasks. More concretely, the paper makes the following
contributions: 
\begin{enumerate}[leftmargin=*, parsep=0pt, itemsep=2pt, topsep=-4pt]
  \item It shows that the existence of co-optimal or approximated 
  solutions obtained by solving hard CO problems to construct training 
  datasets challenges the learnability of the task.
  
  \item To overcome this limitation, it introduces the problem of 
  optimal dataset design, which is cast as a bilevel optimization 
  problem. The optimal dataset design problem is motivated using
  theoretical insights on the approximation of functions by neural 
  networks, relating the properties of a function describing a training 
  dataset to the model capacity required to represent it.
  
  \item It introduces a tractable algorithm for the generation of 
  datasets that are amenable to learning, and empirical demonstration 
  of marked improvements to the accuracy of trained models, as well as 
  the ability to satisfy constraints at inference time. 

  \item Finally, it provides state-of-the-art accuracy results at vastly 
  enhanced computational runtime on learning two challenging optimization 
  problems: Job Shop Scheduling problems and Optimal Power Flow problems 
  for energy networks. 
\end{enumerate}
To the best of the authors knowledge this work is the first to highlight 
the issue of learnability in the face of co-optimal or approximate 
solutions obtained to generate training data for learning to approximate 
hard CO problems. 
{\em The observations raised in this work may result in a broader impact 
as, in addition to approximating hard optimization problems, the optimal 
dataset generation strategy introduced in this paper may be useful to 
the line of work on integrating CO as differentiable layers for predictive 
and prescriptive analytics, as well as for physics constrained learning 
problems such as when approximating solutions to systems of 
partial differential equation.}

\section{Related work}
\label{sec:related_wrok}

The integration of CO models into ML pipelines to meld prediction and decision models has recently seen the introduction of several nonoverlapping approaches, surveyed by \citet{kotary2021end}. 
The application of ML to boost performance in traditional search methods through branching rules and enhanced heuristics is broadly overviewed by \citet{bengio2020machine}. Motivated also by the need for fast approximations to combinatorial optimization problems, the training of surrogate models via both supervised and reinforcement learning is an active area for which a thorough review is provided by \citet{vesselinova2020learning}.  
Designing surrogate models with the purpose to approximate hard CO
problem has been studied in a number of works, including
\cite{borghesi2020combining,detassis2020teaching,Fioretto:AAAI-20}. 
An important aspect for the learned surrogate
models is the prediction of solutions that satisfy the problem
constraints. While this is a very difficult task in general, several methodologies
have been devised. In particular, Lagrangian loss functions have been
used for encouraging constraint satisfaction in several applications
of deep learning, including fairness
\cite{tran:20_fairdp} and energy problems \cite{Fioretto:AAAI-20}. 
Other methods iteratively modify training labels to encourage satifaction of 
constraints during training \cite{detassis2020teaching}.

This work focuses on an orthogonal direction with respect to the
literature reviewed above. Rather than devising a new methodology for
effectively producing a surrogate model that approximate some hard
optimization problem, it studies the machine learning task from a data generation
perspective. It shows that the co-optimality and symmetries of a hard
CO problem may be viewed as a source of noise that obscures the
relationships between training data and the target outputs and
proposes an optimal data generation approach to mitigate these
important issues.

\section{Preliminaries}
\label{sec:prelim}

A constrained optimization (CO) problem poses the task of minimizing 
an \emph{objective function} \mbox{$f:\cY \times \cX \to \mathbb{R}_+$} 
of one or more variables $\bm{y} \in \cY \subseteq \mathbb{R}^n$, 
subject to the condition that a set of \emph{constraints} 
$\mathcal{C}_{\bm{x}}$ are satisfied between the variables and where
$\bm{x} \in\cX \subseteq \mathbb{R}^m$ denotes a vector of input data 
that specifies the problem instance: 
\begin{equation}
\label{eq:opt}
\cO(\bm{x}) = \argmin_{\bm{y}} f(\bm{y}, \bm{x}) \;\; 
  \text{subject to:} \;\;
  \bm{y} \in \mathcal{C}_{\bm{x}}.
\end{equation}
An assignment of values $\bm{y}$ which satisfies $\mathcal{C}_{\bm{x}}$ 
is called a \emph{feasible solution}; if, additionally 
$f(\bm{y}, \bm{x}) \leq f(\bm{w}, \bm{x})$ for all feasible
 $\bm{w}$, it is called an \emph{optimal solution}.

A particularly common constraint set arising in practical problems takes 
the form $ {\cal{C}}= \{ \bm{y}  \;:\; \bm{A} \bm{y} \leq \bm{b} \}$,
where $\bm{A} \in \mathbb{R}^{m \times n}$ and $\bm{b} \in \mathbb{R}^m$. 
In this case, ${\cal{C}}$ is a convex set. If the objective $f$ is an 
affine function, the problem is referred to as \emph{linear program} (LP). 
If, in addition, some subset of a problem's variables are required to 
take integer values, it is called \emph{mixed integer program} (MIP). 
While LPs with convex objectives belong to the class of convex problems, 
and can be solved efficiently with strong theoretical guarantees on the 
existence and uniqueness of solutions \citep{boydconvex},  
the introduction of integral constraints ($\bm{y} \in \mathbb{N}^n$) 
results in a much more difficult problem. The feasible set in MIP 
consists of distinct points in $\bm{y} \in \mathbb{R}^n$, not only 
nonconvex but also disjoint, and the resulting problem is, in general, 
NP-Hard. 
Finally, nonlinear programs (NLPs) are optimization problems where some 
of the constraints or the objective function are nonlinear. Many NLPs 
are nonconvex and can not be efficiently solved \citep{nocedal2006numerical}.

The methodology introduced in this paper is illustrated on hard MIP 
and nonlinear program instances.

\section{Problem setting and goals}

This paper focuses on learning approximate solutions to problem 
\eqref{eq:opt} via supervised learning. The task considers datasets 
$\bm{\chi} = \{(\bm{x}^{(i)}, \bm{y}^{(i)})\}_{i=1}^N$ 
consisting of $N$ data points with $\bm{x}^{(i)} \in \cX$ being a vector 
of input data, as defined in equation \eqref{eq:opt}, and
$\bm{y}^{(i)} \in \cO(\bm{x}^{(i)})$ being a solution of the optimization 
task. A desirable, but not always achievable, property is for the solutions
$\bm{y}^{(i)}$ to be optimal.

The goal is to learn a model $f_\theta : \mathcal{X} \to \mathcal{Y}$, 
where $\theta$ is a vector of real-valued parameters, and whose quality 
is measured in terms of a nonnegative, and assumed differentiable, 
\emph{loss function} $\ell: \mathcal{Y} \times \mathcal{Y} \to \mathbb{R}_+$.  
The learning task minimizes the empirical risk function (ERM):
\begin{equation}
\label{eq:erm}
    \min_\theta J(f_\theta; \bm{\chi}) = 
    \frac{1}{N} \sum_{i=1}^N \ell(f_\theta(\bm{x}^{(i)}), \bm{y}^{(i)}),
\end{equation}
with the desired goal that the predictions also satisfy the problem constraints: $f_\theta(\bm{x}^{(i)}) \in \mathcal{C}_{\bm{x}}$.

While the above task is difficult to achieve due to the presence of constraints, 
the paper adopts a Lagrangian approach \citep{fioretto2020lagrangian}, 
which has been shown successful in learning constrained representations, 
along with projection operators, commonly applied in constrained optimization 
to ensure constraint satisfaction of an assignment. 
For a given point $\hat{\bm{y}}$, e.g., representing the model prediction, 
a projection operator $\pi_{\cC}(\hat{\bm{y}})$ finds the closest feasible 
point $\bm{y} \in \cC$ to $\hat{\bm{y}}$ under a $p$-norm: 
\[ \pi_{\cC}(\hat{\bm{y}}) \defeq
  \argmin_{\bm{y}} \| \bm{y} - \hat{\bm{y}}\|_p \ \;\; 
  \text{subject to:} \;\; \bm{y} \in \mathcal{C}.
\]
The full description of the Lagrangian based approach and the 
projection method adopted is delegated to the Appendix \ref{app:lagrangian}.

\section{Challenges in learning hard combinatorial problems}
\label{sec:challenges}

\begin{figure}[t]
\centering
\includegraphics[width=0.3\linewidth]{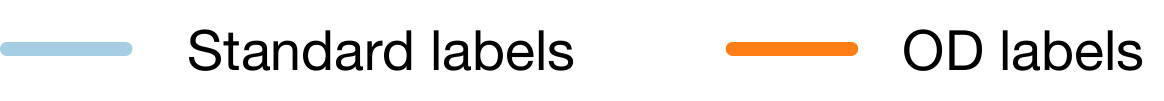}\\
\vspace{-6pt}
\includegraphics[height=80pt]{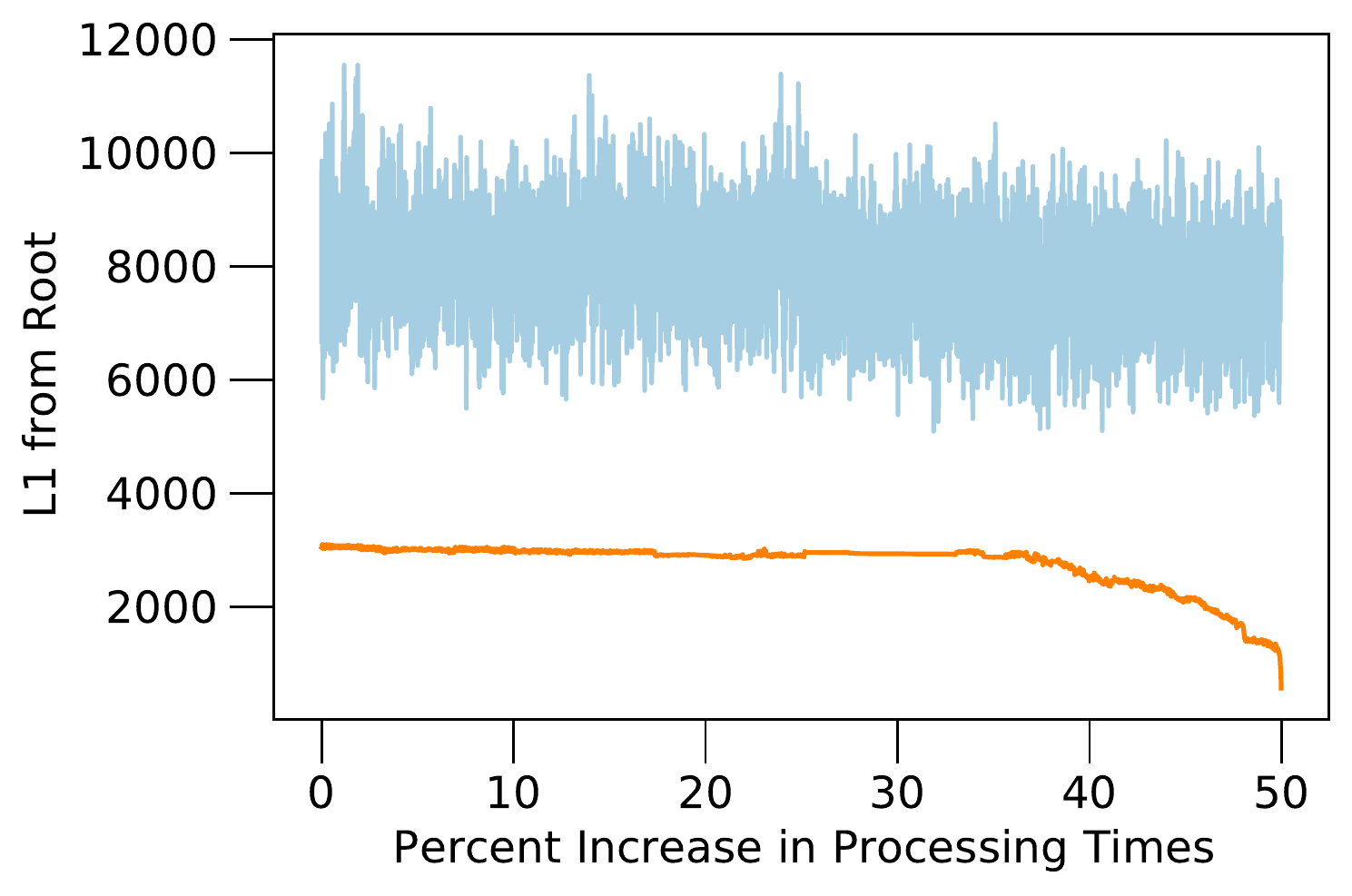}
\hfill
\includegraphics[height=80pt]{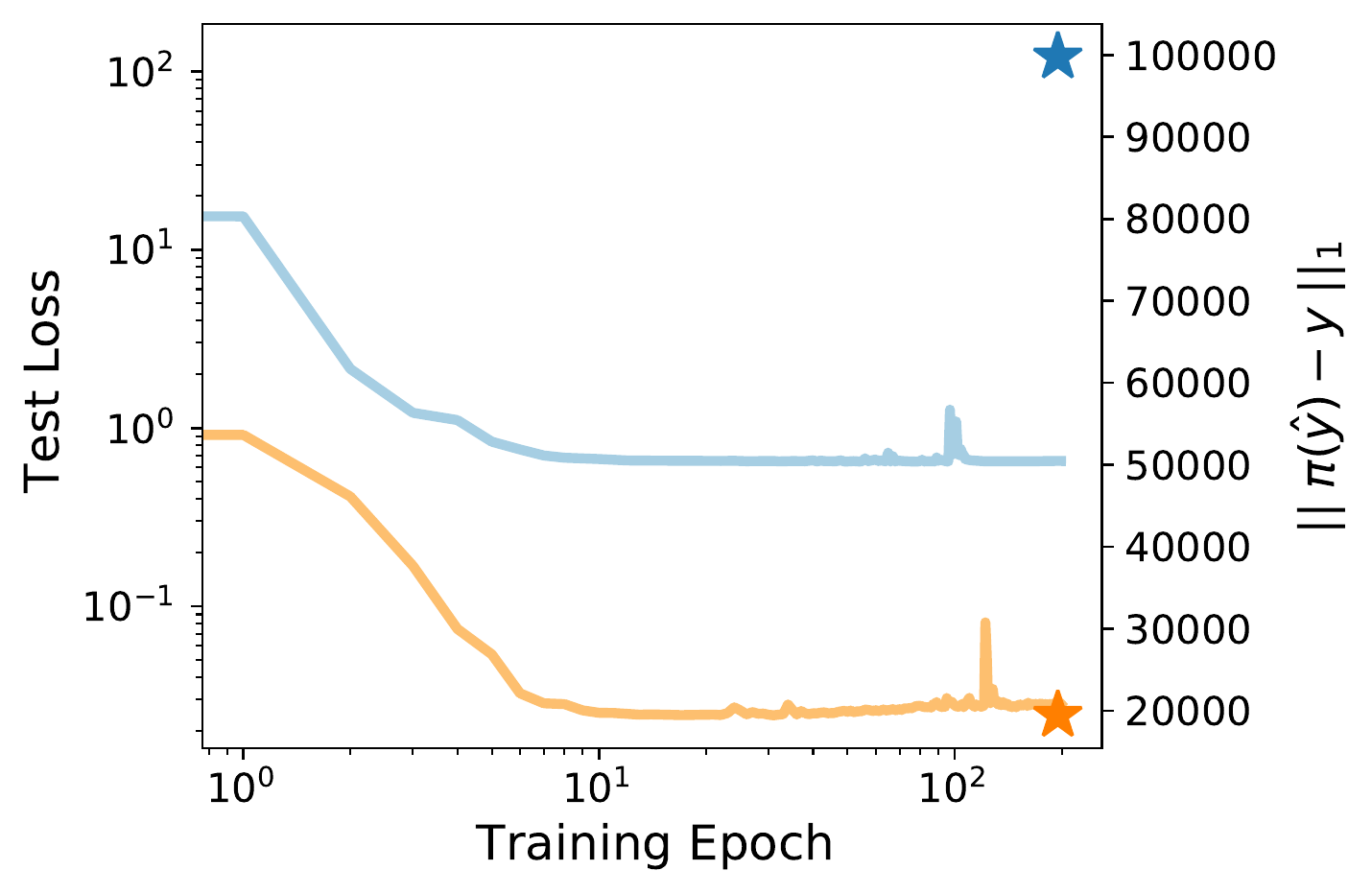}
\hfill
\includegraphics[height=85pt]{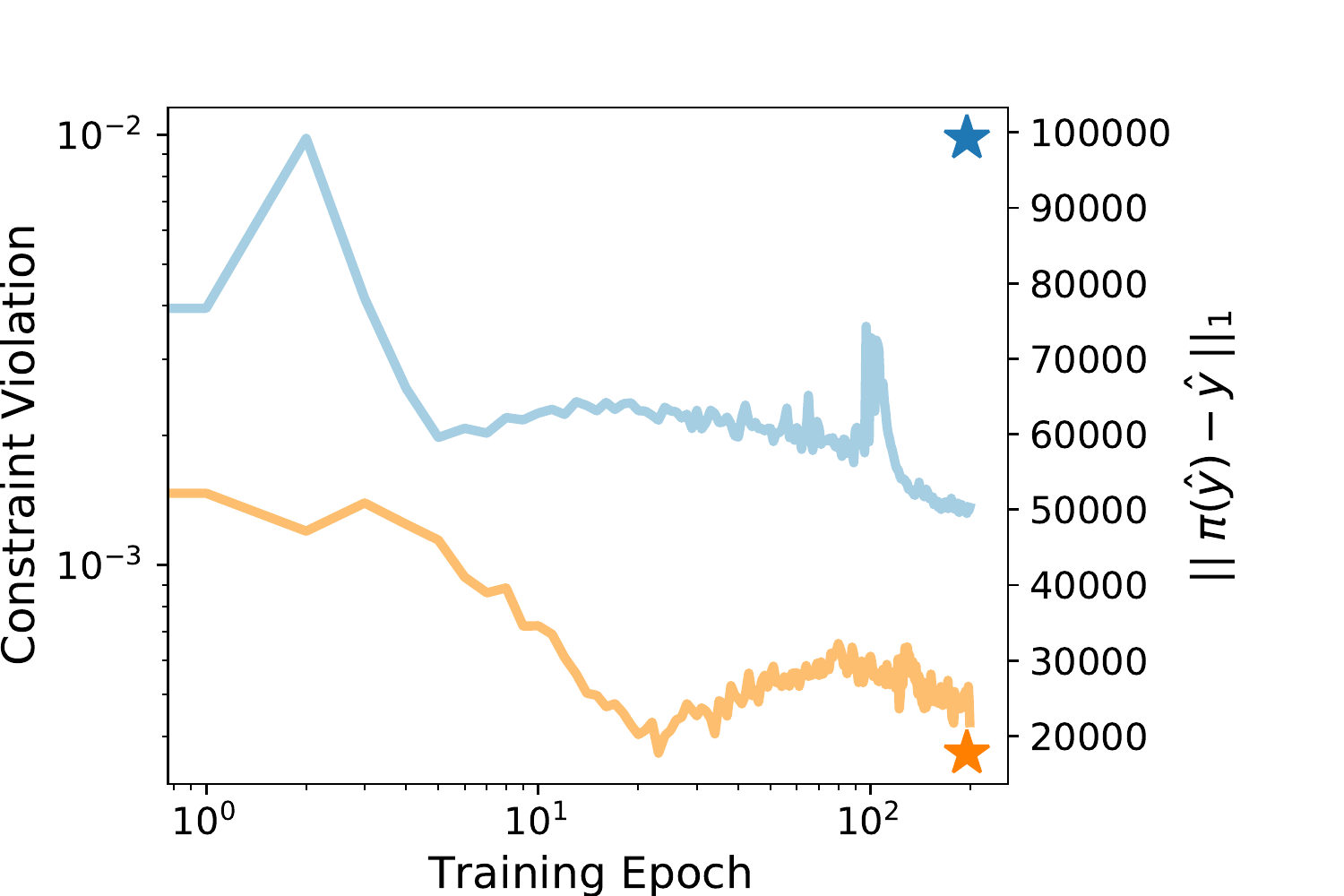}
{\caption{Dataset solution $\bm{y}^{(i)}$ comparison: 
L1 Distance from a reference solution (left); 
Test loss (center); and 
Constraint violations (right). 
}
\label{fig:challenge}}
\end{figure}
One of the challenges arising in this area comes from the recognition that a problem instance may admit a variety of disparate optimal solutions for each input $\bm{x}$. To illustrate this challenge, the paper uses a set of scheduling instances that differ only in the time required to process tasks on some machine. A standard approach to the generation of dataset in this context would consist in solving each instance independently using some SoTA optimization solver. However, this may create some significant issues that are illustrated in Figure \ref{fig:challenge} (more details on the problem are provided in Section \ref{sec:jss}). The blue curve
in Figure \ref{fig:challenge} (left) illustrates the behavior of this natural approach. In the figure, the processing times in the instances increase from left to right and the blue curve represents the $L_1$-distance between the obtained solution to each instance (i.e., the start times of the tasks) and a reference optimal solution for some instance. The volatile curve shows that meaningful patterns 
can be lost, including the important relationship between an increase in processing times and the resulting solutions. 
Figure \ref{fig:challenge} (center) shows that, while the solution patterns induced by the target labels appear volatile, the ERM  problem appears well behaved, in the face of minimizing the test loss. 
However, when training loss converges, accuracy (measured as the distance 
between the projection of the prediction $\pi_{\cC}(\hat{\bm{y}})$ and the real label $\bm{y}$) remains poor in models trained on the such data (blue star).  Figure \ref{fig:challenge} (right) shows the average magnitude of the constraints violation during training, corresponding to the two target solution sets of Figure \ref{fig:challenge} (left), along with a comparison of the objective of the projection operator applied to the prediction: $\| \pi_{\cC}( \hat{\bm{y}} ) - \hat{\bm{y}} \|$. It is worth emphasizing that these volatility issues are further exacerbated when time constraints prevent the solver from obtaining optimal solutions. Moreover, similar patterns can also be observed for the data generated while solving optimal power flow instances that exhibit symmetries. 

\begin{wrapfigure}[8]{r}{130pt}
\vspace{-28pt}
\includegraphics[width=\linewidth]{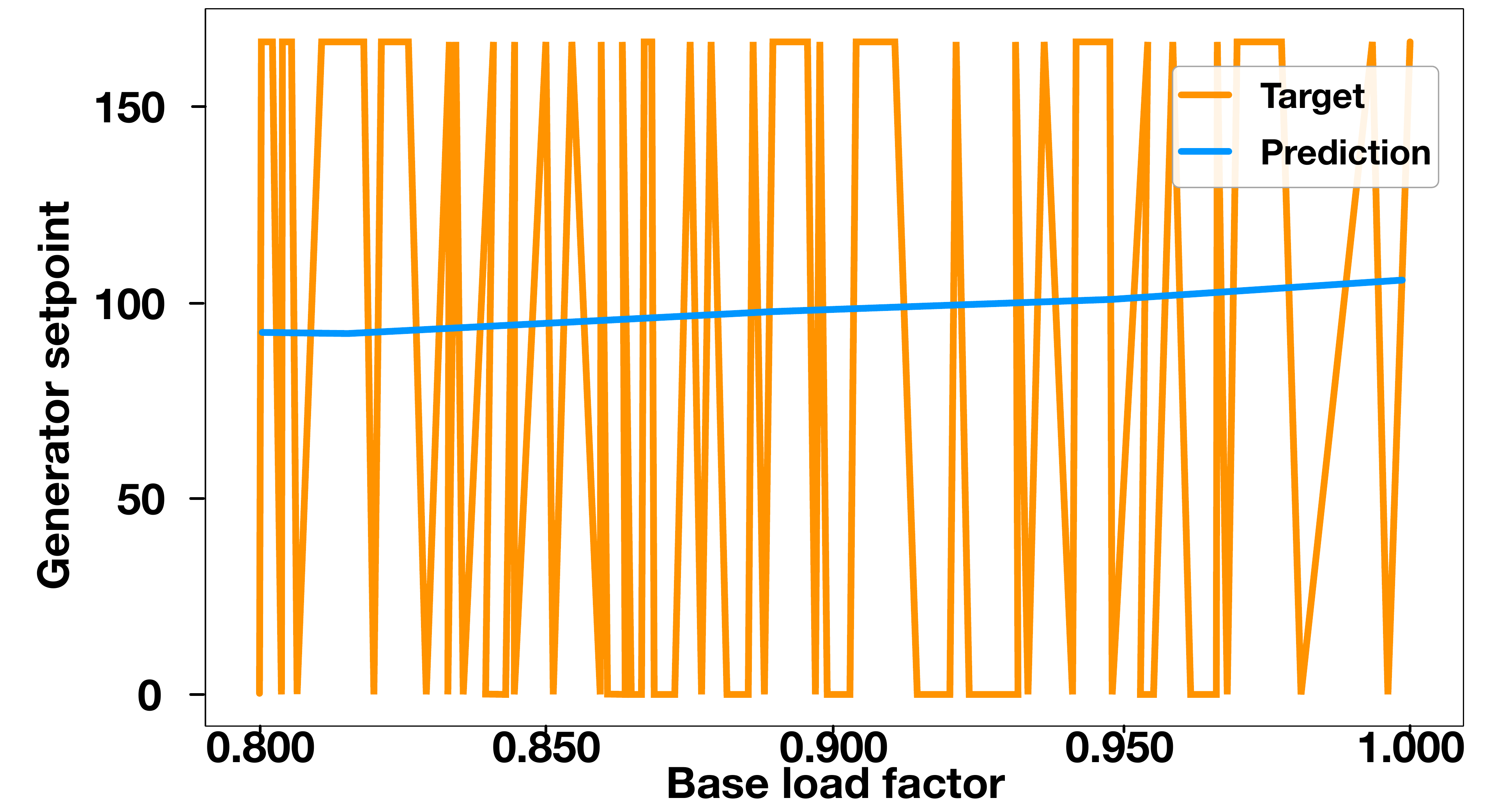}
\vspace{-18pt}
\caption{\small Approximating highly volatile function results in 
low-variance models. \label{fig:low_var}}
\end{wrapfigure}
Additionally, extensive observations collected on the motivating 
applications of the paper show that, even when the model complexity 
(i.e., the dimensionality of the model parameters 
$\theta$) is increased arbitrarily, the resulting learned models tend to 
have low-variance. This is illustrated in Figure 
\ref{fig:low_var}, where the orange and blue curves depict, respectively, 
a function interpolating the training labels and the associated learned 
solutions. 

The goal of this paper is to construct datasets that are well-suited
for learning the optimal (or near-optimal) solutions to optimization
problems.  The benefits of such an approach is illustrated by the
orange curves and stars in Figure \ref{fig:challenge}, which were 
obtained using the data generation methodology proposed in this paper. 
They considered the same instances and obtained (different) optimal 
solutions, but exhibit much enhanced behavior on all metrics. 

\section{Theoretical justification of the data generation}
\label{sec:theory}

Although the data generation strategy is necessarily heuristic, it
relies on key theoretical insights on the nature of optimization
problems and the representation capabilities on neural networks. This
section reviews these insights.

First, observe that, as illustrated in the motivating Figure 
\ref{fig:low_var}, the solution trajectory associated with the problem
instances on various input parameters can often be naturally
approximated by piecewise linear functions. This approximation is in
fact exact for linear programs when the inputs capture incremental
changes to the objective coefficients or the right-hand side of the
constraints. 
Additionally, ReLU neural networks, used in this paper to approximate
the optimization solutions, have the ability to capture piecewise
linear functions \cite{huang2020relu}. While these models are thus 
compatible with the task of predicting the solutions
of an optimization problem, the model capacity required to represent
a target piecewise linear function exactly depends directly on the
number of constituent pieces.

\begin{theorem}[Model Capacity \citep{arora2016understanding}]
Let $f : \RR^d \to \RR$ be a piecewise linear function with $p$ pieces. If $f$ is represented by a ReLU network with depth $k + 1$, then it must have size at least $\frac{1}{2}k p^{\frac{1}{k}}-1$. Conversely, any piecewise linear function $f$ that is represented by a ReLU network of depth $k + 1$ and size at most $s$, can have at
most $\left( \frac{2s}{k} \right)^k$ pieces.
\label{arora_lemma_D6}
\end{theorem}

The solution trajectories may be significantly different depending on how the data is generated. Hence, the more volatile the trajectory, the harder it will be to learn. Moreover, for a network of fixed size, the more volatile the trajectory, the larger the approximation error will be in general. The data generation proposed in this paper will aim at generating solution trajectories that are approximated by small piecewise linear functions. The following theorem bounds the approximation error when using continuous piecewise linear functions: it connects the approximation errors of a piecewise linear function with the \emph{total variation in its slopes}. 

\begin{theorem}
\label{thm_james}
Suppose a piecewise linear function $f_{p'}$,  with $p'$ pieces each of width $h_k$ for $k \!\in\! [p']$, is used to approximate a piecewise linear $f_p$ with $p$ pieces, where $p' \!\leq\! p$. Then the approximation error 
$$\|f_p - f_{p'} \|_1 \leq \frac{1}{2}   h_{\max}^2  \sum_{1 \leq k \leq p} | L_{k+1} - L_k |,$$
holds where $L_k$ is the slope of $f_p$ on piece $k$ and $h_{\max}$ is the maximum width of all pieces. 
\end{theorem}  
\begin{proof}
Firstly, the proof proceeds with considering the special case 
in which $f_p$ conincides in slope and value with $f_{p'}$ at some point, 
and that each piece of $f_{p'}$ overlaps with at most $2$ distinct 
pieces of $f_{p}$. This is always possible when $p' \geq \frac{p}{2}$. 
Call $I_{k}$ the interval on which $f_{p'}$ is defined by its $k^{th}$ 
piece. If  $I_{k}$ overlaps with only one piece of $f_p$, then for $x \in I_{k}$,
\begin{equation}
  |f_p(x) - f_{p'}(x) | = 0
\end{equation}
If $I_{k}$ overlaps with pieces $k$ and $k+1$ of $f_p$, then for $x \in I_{k}$,
\begin{equation}
|f_p(x) - f_{p'}(x) | \leq  h_k | L_{k+1} - L_k |
\end{equation}
Each of the above follows from the assumption that $f_p $ and $ f_{p'}$ 
are equal in their slope and value at some point within  $I_{k}$. 
From this it follows that on $I_k$, 
\begin{equation}
\|f_p - f_{p'} \|_1 = \int_{I_k}  |f_p - f_{p'} |  \leq  \frac{1}{2} \sum_{1 \leq i \leq p}  h_k^2 | L_{k+1} - L_k | \leq   \frac{1}{2}   h_{\max}^2  \sum_{1 \leq k \leq p} | L_{k+1} - L_k |,
\end{equation}
so that on the entire domain of $f_p$ and $f_{p'}$, 
\begin{equation}
\|f_p - f_{p'} \|_1 \leq \frac{1}{2}   h_{\max}^2  \sum_{1 \leq k \leq p} | L_{k+1} - L_k |.
\end{equation}
Since removing the initial simplifying assumptions tightens this upper 
bound, the result holds.
\label{proof_james}
\end{proof}

The final observation that justifies the proposed approach is the fact
that optimization problems typically satisfy a local Lipschitz
condition, i.e., if the inputs of two instances are close, then they
admit solutions that are close as well, i.e., 
for $\optimal{\bm{y}}^{(i)} \in \cO(\bm{x}^{(i)})$ and
 $\optimal{\bm{y}}^{(j)} \in \cO(\bm{x}^{(j)})$, 
\begin{equation}
\label{eq:discrete_L}
 \| \optimal{\bm{y}}^{(i)} - \optimal{\bm{y}}^{(j)} \|  \leq 
 C \| \bm{x}^{(i)} - \bm{x}^{(j)} \|, 
 \end{equation} 
for some $C \geq 0$ and $\| \bm{x}^{(i)} - \bm{x}^{(j)} \| \leq \epsilon$,
where $\epsilon$ is a small value.
This is obviously true in linear programming
when the inputs vary in the objective coefficients or the right-hand
side of the constraints, but it also holds locally for many other
types of optimization problems. That observation suggests that, when
this local Lipschitz condition holds, it may be possible to generate
solution trajectories that are well-behaved and can be approximated
effectively. Note that Lipschitz functions can be nicely approximated
by neural networks as the following result indicates.

\begin{theorem}[Approximation \citep{chong2020closer}]
If $f:[0,1]^n \rightarrow \mathbb{R}$ is $L$-Lipschitz continuous, then for every $\epsilon > 0$, there exists some single-layer neural network $\rho$ of size $N$ such that 
\(\|f - \rho \|_{\infty} < \epsilon\),
where 
\(N = \binom{n+ \frac{3L}{\epsilon}}{n}.\)

\label{Cong_Corr_4.1}
\end{theorem}
The result above illustrates that the model capacity required to approximate a given function depends to a non-negligible extent on the Lipschitz constant value of the underlying function.

Note that the results in this section are bounds on the ability of 
neural networks to represent generic functions. In practice, these 
bounds themselves rarely guarantee the training good approximators, 
as the ability to minimize the empirical risk problem in practice is 
often another significant source of error. In light of these results 
however, it is to be expected that datasets which exhibit less variance 
and have small Lipschitz constants\footnote{The notation here is used to 
denote its discrete equivalent, as indicated in Equation \eqref{eq:discrete_L}.}
will be better suited to learning 
good function approximations. The following section presents a method 
for dataset generation motivated by these considerations. 

\newcommand{\bx}[1]{\bm{x}^{(#1)}}
\newcommand{\by}[1]{\bm{y}^{(#1)}}

\section{Optimal CO training data design}
\label{sec:optimal_dataset}

Given a set of input data $\{\bx{i}\}_{i=1}^N$, the goal is to construct the associated pairs $\by{i}$ for each $i\in [N]$, that solve the following problem
\begin{subequations}
\begin{align}
  \min_{\theta, \bm{y}^{(i)}} &\frac{1}{N} \sum_{i=1}^N \ell(f_\theta(\bx{i}), \by{i})
  \\
  \st:& \;\; \by{i} \in \argmin_{\bm{y} \in \mathcal{C}_{\bx{i}}} f(\bm{y}, \bx{i}).
\end{align}
\end{subequations}
One often equips the data point set $\{\bx{i}\}_{i=1}^N$ with an ordering relation $\preceq$ 
such that $\bm{x} \preceq \bm{x}' \Rightarrow \| \bm{x} \|_p \leq \| \bm{x}'\|_p$ for some $p$-norm. 
For example, in the scheduling domain, the data points $\bm{x}$ represent task start times and the training data are often generated by ``slowing down'' some machine, which simulates some unexpected ill-functioning component in the scheduling pipeline. In the energy domain, $\bm{x}$ represent the load demands and the training data are generated by increasing or decreasing these demands, simulating the different power load requests during daily operations in a power network.

From the space of co-optimal solutions $\by{i}$ to each problem
instance $\bx{i}$, the goal is to generate solutions which coincide,
to the extent possible, with a target function of low total variation
and Lipschitz factor, as well as a low number of constituent linear
pieces in the case of discrete optimization. While it may not be
possible to produce a target set that simultaneously optimizes each
of these metrics, they are confluent and can be improved
simultaneously. Natural heuristics are available which reduce these 
metrics substantially when compared with naive approaches. 

One heuristic aimed at satisfying the aforementioned properties
reduces to the problem of determining a solution set $\{\by{i}\}_
{i=1}^N$ for the inputs $\{\bx{i}\}_{i=1}^N$ of problem \eqref
{eq:opt} that minimizes their total variation:
\begin{subequations}
\begin{align}
  \minimize \textsl{TV}\left(\{\by{i}\}_{i=1}^N\right) =
  &\frac{1}{2} \sum_{i=1}^{N-1}
    \| \by{i+1} - \by{i} \|_p  \label{min_TV_obj}  \\ 
    \st:& \;\; \by{i} = \argmin_{\bm{y} \in \mathcal{C}_{\bx{i}}} f(\bm{y}, \bx{i}).
\end{align}
\label{TV_dataset_def}
\end{subequations}

\begin{wrapfigure}[11]{r}{160pt}
  \begin{algorithm}[H]
  {\small
    \DontPrintSemicolon
    \caption{Opt.~Data Generation}
    \label{alg:seqsolve}
    \setcounter{AlgoLine}{0}
    \SetKwInOut{Input}{input}
  
    \Input{$\{\bm{x}^{(i)}\}_{i=1}^N$: Input data}
    \label{line:1a}
      $\bm{y}^{(N)} \gets \optimal{\bm{y}}^{(N)} \in \tilde{\cO}(\bm{x}^{(N)})$\;
      \For{$i = N-1$ down to $1$}
     {
        $
        \bm{y}^{(i)} \in
          \begin{cases}
          \argmin_{\bm{y}} \hspace{10pt}  \| \bm{y} - \bm{y}^{(i+1)}\|_p\\
          \text{subject to:} \hspace{4pt}  \bm{y} \in \mathcal{C}_{\bm{x^{(i)}}} \\
                            \hspace{40pt} f(\bm{y}) \leq f(\optimal{\bm{y}}^{(i)})\!\!\!\!\!\!\!\!\!\!\!\!
         \end{cases}
      $
    }
    \Return $\bm{\chi} = \left\{\left(\bm{x}^{(i)}, \bm{y}^{(i)}\right)\right\}_{i=1}^N$
  }
  \end{algorithm}
\end{wrapfigure}  
In practice, this bi-level minimization cannot be
 achieved, due partially to its prohibitive size. It is possible,
 however, to minimize the individual terms of \eqref
 {min_TV_obj}, each subject to the result of the previous, by solving
 individual instances sequentially.    Algorithm \ref
 {alg:seqsolve} ensures that solutions to subsequent instances have
 minimal distance with respect to the chosen $p$-norm(the experiments
 of Section \ref{experimental_section} use $p=1$). This method
 approximates a set of solutions with minimal total variation, while
 ensuring that the maximum magnitude of change between subsequent
 instances is also small. When the data represent the result of a
 discrete optimization, this coincides naturally with a
 representative function which requires less pieces, with less
 extreme changes in slope. The method starts by solving a target
 instance, e.g., the last in the given ordering $\preceq$ (line 1).
 Therein, $\tilde{\cO}$ denotes the solution set of a (possibly approximated) minimizer
 for problem \eqref{eq:opt}. In the case of the job shop scheduling,
 for example, $\tilde{\cO}$ represents a local optimizer with a
 suitable time-limit. The process then generates the next dataset
 instance $\bm{y}^{(i)}$ in the ordering $\preceq$ by solving the
 optimization problem given in line (3). The problem finds a solution
 to problem $\bm{x}^{(i)}$ that is close to adjacent solution $\bm
 {y}^{(i+1)}$ while preserving \emph{optimality}, i.e., the objective
 of the sought $\bm{y}$ is constrained to be at most that of $\optimal{\bm{y}}^{(i)} \in
 \tilde{\cO}(\bm{x}^{(i)})$.

The method hinges on the assumption that a bound on 
$\optimal{\bm{y}}^{(i)}$ can be inferred from $\optimal{\bm{y}}^{(i+1)}$, 
as in the case studies analyzed in this paper. In addition, 
$\optimal{\bm{y}}^{(i+1)}$ may be used to hot-start the solution of 
the subsequent problem, carrying forward solution progress between 
iterations of Algorithm \ref{alg:seqsolve}. Therefore, a byproduct 
of this data-generation approach is that the optimization problem in 
line (3) can be well-approximated within a short timeout, resulting 
in enhanced efficiency which makes the data generation process viable 
in practice even for hard optimization problems that require substantial 
solving time when treated independently.

In addition to providing enhanced efficacy for learning, this method of generating target instances is generally preferable from a modeling point of view. When predicted solutions to related decision problems are close together, the resulting small changes are often more practical and actionable, and thus highly preferred in practice. {For example, a small change in power demands should result in an updated optimal power network configuration which is easy to achieve given its previous state.}

\section{Application to case studies}
\label{experimental_section}
\label{sec:jss} 

The concepts introduced above are applied in this section to two 
representative case studies, \emph{Job Shop Scheduling} (JSS) and 
\emph{Optimal Power Flow} (OPF). Both are of interest in the
 optimization and machine learning communities as practical problems
 which must be routinely solved, but are difficult to approximate
 under stringent time constraints. The JSS problem represents the
 class of combinatorial problems, while the OPF problem is continuous
 but nonlinear and non convex. Both lack solution methods with strong 
 guarantees on the rate of convergence, and the quality of solutions that 
 can be obtained. 
 In the studies described below, a deep neural ReLU network equipped 
 with a Lagrangian loss function (described in details in Appendix 
 \ref{app:lagrangian}) is used to predict the problem solutions that are 
 approximately feasible and close to optimal. Efficient projection 
 operators are subsequently applied to ensure feasibility of the 
 final output (See Appendix \ref{app:jss} and \ref{app:opf}).

\subsubsection*{Job shop scheduling}
Job Shop Scheduling (JSS) assumes a set of $J$ jobs, each
 consisting of a list of $M$ tasks to be completed in a specified
 order. Each task has a fixed processing time and is assigned to one
 of $M$ machines, so that each job assigns one task to each machine.
 The objective is to find a schedule with minimal \emph{makespan}, or
 time taken to process all tasks. The \emph{no-overlap} condition
 requires that for any two tasks assigned to the same machine, one
 must be complete before the other begins. See the problem
 specification in Appendix  \ref{app:jss}.
 The objective of the learning task is to predict the start times of 
 all tasks given a JSS problem specification (task duration, machine assignments). 

\textbf{Data Generation Algorithms}
The experiments examine the proposed models on a variety of problems 
from the JSPLIB library \cite{tamy0612_2014}. The ground truth data are constructed as follows: different input data $\bm{x}^{(i)}$ are generated by simulating a machine slowdown, i.e., by altering the time required to process the tasks on that machine by a constant amount which depends on the instance $i$. Each training dataset associated with a JSS benchmark is composed of a total of 5000 instances. 
Increasing the processing time of selected tasks may also change the 
difficulty of the scheduling. 
The method of sequential solving outlined in Section \ref{sec:optimal_dataset} is particularly well-suited to this context.
 Individual problem instances can be ordered relative to the amount
 of extension applied to those processing times, so that when $d^{
 (i)}_{jt}$ represents the time required to process task $t$ of job
 $j$ in instance $i$, $ d^{(i)}_{jt} \leq d^{(i+1)}_
 {jt} \; \; \forall j,t$. In this case, any solution to instance $d_{
 (i+1)}$ is feasible to instance $d_i$ (tasks in a feasible schedule
 cannot overlap when their processing times are reduced, and start
 times are held constant). As such, the method can be made efficient
 by passing the solution between subsequent instances as a
 hot-start. 

The analysis compares two datasets: One consisting of target solutions generated independently with a solving time limit of $1800$ seconds using the state-of-the-art IBM CP Optimizer constraint programming software (denoted as Standard), and one whose targets are generated according to algorithm \ref{alg:seqsolve}, called the Optimal Design dataset (denoted as OD). 

\begin{wrapfigure}[9]{r}{150pt}
\vspace{-15pt}
\resizebox{\linewidth}{!}
{
\begin{tabular}{rl ll}
\toprule
  Instance & \multicolumn{1}{c}{Size} 
           & \multicolumn{2}{c}{Total Variation ($\times 10^6$)}\\
\cmidrule(r){3-4} 
& $J\times M$ & Standard Data & OD Data \\
\midrule
ta25   & $20 \!\times\! 20$ & $67.8 $ & $\bm{0.194}$ \\ 
yn02   & $20 \!\times\! 20$ & $55.0 $ & $\bm{0.483}$ \\ 
swv03  & $20 \!\times\! 10$ & $109.4$ & $\bm{0.424}$\\ 
swv07  & $20 \!\times\! 15$ & $351.2$ & $\bm{0.100}$\\ 
swv11  & $50 \!\times\! 10$ & $352.0$ & $\bm{1.376}$\\
\bottomrule
\end{tabular}
}
  \caption{Standard vs OD training data: Total Variation.}
  \label{tab:JSS_tv}
  \vspace{-12pt}
\end{wrapfigure} Figure \ref{tab:JSS_tv} presents a comparison of the
 total variation resulting from the two datasets. Note that the OD
 datasets have total variation which is orders of magnitude lower
 than their Standard counterparts. Recall that a small total
 variation is perceived as a notion of well-behaveness from the perspective  of
function approximation. Additionally, it is noted that the total computation
 time required to generate the OD dataset is at least an order of
 magnitude smaller than that required to generate the standard
 dataset ($13.2$h vs.~$280$h). 
\begin{table}[!tb]
\centering
\resizebox{0.8\linewidth}{!}
{
\begin{tabular}{rl ll ll ll ll ll}
\toprule
Instance & \multicolumn{1}{c}{Size} 
         & \multicolumn{2}{c}{Prediction Error} 
         & \multicolumn{2}{c}{Constraint~Violation} 
         & \multicolumn{2}{c}{Optimality Gap (\%)} 
         & \multicolumn{2}{c}{Time SoTA Eq.~(s)}\\    
\cmidrule(r){3-4} 
\cmidrule(r){5-6}
\cmidrule(r){7-8}
\cmidrule(r){9-10} 
  & $J\times M$
  & Standard & OD 
  & Standard & OD 
  & Standard & OD
  & Standard & OD \\
\midrule                                                                   
ta25   & $20 \!\times\! 20$ & $193.9$ & $\bm{23.4}$ & $180.0$ & $\bm{45.5}$ & $10.3$ & $\bm{4.0}$  & $24$ & $\bm{550}$\\ 
yn02   & $20 \!\times\! 20$ & $153.2$ & $\bm{38.9}$ & $124.9$ & $\bm{70.3}$ & $9.1 $ & $\bm{4.5}$  & $27$ & $\bm{45 }$ \\ 
swv03  & $20 \!\times\! 10$ & $309.4$ & $\bm{12.4}$ & $206.9$ & $\bm{31.6}$ & $18.0$ & $\bm{2.2}$  & $15$ & $\bm{65 }$  \\ 
swv07  & $20 \!\times\! 15$ & $330.4$ & $\bm{19.9}$ & $280.1$ & $\bm{67.2}$ & $17.0$ & $\bm{3.0}$  & $15$ & $\bm{60 }$ \\ 
swv11  & $50 \!\times\! 10$ & $1090.0$ & $\bm{51.2}$ & $906.4$ & $\bm{151.7}$ & $28.5$ & $\bm{4.5}$  & $13$ & $\bm{100}$\\ 
\bottomrule
\end{tabular}
}
  \caption{Standard vs OD training data: prediction errors, constraint violations, and 
  optimality gap (the smaller the better), Time SoTA Eq.~(the larger the better). 
  Best results are highlight in bold.}
    \label{tab:JSS_acc}
\end{table}

\textbf{Prediction Errors and Constraint Violations}
Table \ref{tab:JSS_acc} reports the prediction errors 
as $L_1$-distance between the (feasible) predicted variables, 
i.e., the projections $\pi(\hat{\bm{y}})$ and their original ground-truth 
quantities ($\bm{y}$), the average constraint violation degrees, expressed 
as the $L_1$-distance between the predictions  and their  projections, and 
the optimality gap, which is the relative difference in makespan (or,
equivalently objective values) between the predicted (feasible) schedules 
and target schedules. All these metrics are averaged over all perturbed 
instances of the dataset and expressed in percentage. In the case of the former two metrics, values are reported as a percentage of the average task duration per individual instance. 
Notice that for all metrics the methods trained using the OD datasets
result in drastic improvements (i.e., one order of magnitude) with 
respect to the baseline method. 
Additionally, Table \ref{tab:JSS_acc} (last column) reports the runtime
required by CP-Optimizer to find a value with the same makespan as the 
one reported by the projected predictions (projection times are also included). 
The values are to be read as the larger the better, and present a remarkable
improvement over the baseline method.
It is also noted that the worst average time required to obtain a feasible 
solution from the predictions is 0.02 seconds.
Additional experiments, reported in Appendix \ref{app:results} also show 
that the observations are robust over a wide range of hyper-parameters
adopted to train the learning models.

{\em The results show that the OD data generation can drastically improve 
predictions qualities while reducing the effort required by a 
projection step to satisfy the problem constraints.}

\subsubsection*{AC Optimal Power Flow}
\label{sec:opf}

\newcommand{\OPF}{{\textsl{OPF}}}
\newcommand{\PF}{\textsl{PF}}
\newcommand{\sC}{{\mathscr{C}}}

\emph{Optimal Power Flow (OPF)} is the problem of finding the best
generator dispatch to meet the demands in a power network. The OPF is 
solved frequently in transmission systems around the world 
and is increasingly difficult due to intermittent renewable energy sources. 
The problem is required to satisfy the AC power flow equations, that 
are non-convex and nonlinear, and are a core building block in 
many power system applications.
The objective function captures the cost of the generator 
dispatch, and the Constraint set describes the power flow operational 
constraints, enforcing generator output, line flow limits, Kirchhoff's 
Current Law and Ohm's Law for a given load demand. 
The OPF receives its input from unit-commitment algorithms that 
specify which generators will be committed to deliver energy and 
reserves during the 24 hours of the next day. Because many generators 
are similar in nature (e.g., wind farms or solar farms connected to 
the same bus), the problem may have a large number of symmetries. If 
a bus has two symmetric generators with enough generator capacities, 
the unit commitment optimization may decide to choose one of the 
symmetric generators or to commit both and balance the generation 
between both of them.
The objective of the learning task is to predict the generator 
setpoints (power and voltage) for all buses given the problem inputs 
(load demands).

\textbf{Data Generation Algorithms}
The experiments compare this commitment strategy and its 
effect on learning on Pegase-89, which is a coarse aggregation of the 
French system and IEEE-118 and IEEE-300, from the NESTA library 
\cite{Coffrin14Nesta}. All base instances are solved using the Julia 
package PowerModels.jl \cite{Coffrin:18} with the nonlinear solver 
IPOPT \cite{wachter06on}. Additional data is reported in Appendix \ref{app:opf}. A number of renewable generators are duplicated at each 
node to disaggregate the generation capabilities. The test cases vary 
the load data by scaling the (input) loads from 0.8 to 1.0 times 
their nominal values. Instances with higher load pattern are typically infeasible. 
The unit-commitment strategy sketched above 
can select any of the symmetric generators at a given bus (Standard data). 
The optimal objective values for a given load are the same, 
but the optimal solutions vary substantially. Note that, when the 
unit-commitment algorithm commits generators by removing symmetries (OD data), the solutions for are typically close to each other when the loads are close. As a result, they naturally correspond to the generation procedure advocated in this paper. 

\begin{table}[!tb]
\centering
\resizebox{0.75\linewidth}{!}
{
\begin{tabular}{rl ll ll ll ll}
\toprule
Instance & \multicolumn{1}{c}{Size} 
         & \multicolumn{2}{c}{Prediction Error} 
         & \multicolumn{2}{c}{Constraint~Violation} 
         & \multicolumn{2}{c}{Optimality Gap (\%)}\\
\cmidrule(r){3-4} \cmidrule(r){5-6} \cmidrule(r){7-8}
  & No.~buses
  & Standard & OD 
  & Standard & OD 
  & Standard & OD \\
\midrule 
 Pegase-89 & 89  & $89.17$ & $\bm{2.78}$ & $1.353$ & $\bm{0.003}$ & $20.1$ & $\bm{0.83}$ \\
 IEEE-118  & 118 & $36.55$ & $\bm{0.54}$ & $1.330$ & $\bm{0.002}$ & $3.80$ & $\bm{0.36}$ \\
 IEEE-300  & 300 & $157.3$ & $\bm{2.27}$ & $1.891$ & $\bm{0.009}$ & $22.9$ & $\bm{0.12}$ \\
\bottomrule
\end{tabular}
}
  \caption{Standard vs OD training data: prediction errors, constraint violations, and 
  optimality gap.}
    \label{tab:OPF_acc}
\end{table}

\textbf{Prediction Errors and Constraint Violations}
As shown in Table \ref{tab:OPF_acc}, when compared the OD approach to
data generation results in predictions that are closer to their 
optimal target solutions (error expressed in MegaWatt (MW)), reduce the 
constraint violations (expressed as $L_1$-distance between the predictions
and their projections), and improve the optimality gap, which is the 
relative difference in objectives between the predicted (feasible) 
solutions and the target ones. 

\section{Limitations and Conclusions}
\label{sec:limitations}
Aside from the inherent difficulty of learning feasible CO solutions, 
a practical challenge is represented by the 
data generation itself. Generating training datasets for supervised 
learning tasks requires solving many instances of hard CO problems, 
which can be time consuming and imposes a toll on energy usage and CO2 
emissions. In this respect, an advantage of the proposed method, 
is its ability to use hot-starts to generate instances incrementally,
resulting in enhanced efficiency 
even for hard optimization problems that require substantial solving 
time when treated independently.

A challenge posed by the proposed data generation methodology its 
restriction to classes of CO problems that do not necessarily require 
diverse solutions over time. For example, in timetabling applications, 
as in the design of employees shifts, a desired condition may be for 
shifts to be diverse for different but similar inputs. The proposed 
methodology may, in fact, induce a learner to predict similar solutions 
across similar input data. 
A possible solution to this problem may be that of generating various 
\emph{trajectories} of solutions, learn from them with independent models, 
and then randomize the model selection to generate a prediction. 

Aside these limitations, the observations raised in this work may 
be significant in several areas: In addition to approximating hard 
optimization problems, the optimal dataset generation strategy introduced 
in this paper may be useful to the line of work on integrating CO and 
machine learning for predictive and prescriptive analytics, as well as 
for physics constrained learning problems, two areas with significant 
economic and societal impacts.


\section*{Acknowledgments}

This research is partially supported by NSF grant 2007164. Its views and conclusions are those of the authors only and should not be interpreted as representing the official policies, either expressed or implied, of the sponsoring organizations, agencies, or the U.S.~government.

\bibliographystyle{abbrvnat}
\bibliography{references}

\appendix
\pagenumbering{arabic}
\renewcommand{\thepage} {A--\arabic{page}}

\section{Lagrangian Dual-based approach}
\label{app:lagrangian}

In both case studies presented below, a constrained deep learning approach is used which encourages the satisfaction of constraints within predicted solutions by accounting for the violation of constraints in a \emph{Lagrangian} loss function  
\begin{equation}
\label{eq:lagrangian_function2}
f_{\bm{\lambda}}(y) = f(y) + \sum_{i=1}^m \lambda_i \max(0, g_i(y)),
\end{equation}
where $f$ is a standard loss function (i.e., \emph{mean squared error}), $\lambda_i$ are \emph{Lagrange multipliers} and $g_i$ represent the constraints of the optimization problem under the generic representation 

\begin{equation}
\label{eq:problem}
  {\cal P} = \argmin_{y} h(y) \;\;
  \mbox{subject to } \;\; g_i(y) \leq 0 \;\; (\forall i \in [m]).
\end{equation}

Training a neural network to minimize the Lagrangian loss for some value of $\lambda$ is anologous to computing a Lagrangian Relaxation:

\begin{equation}
\label{eq:LR}
LR_{\bm{\lambda}} = \argmin_y f_{\bm{\lambda}}(y),
\end{equation}

and the \emph{Lagrangian Dual} problem maximizes the relaxation over all possible $\lambda$:

\begin{equation}
\label{eq:LD}
LD = \argmax_{\bm{\lambda} \geq 0} f(LR_{\bm{\lambda}}).
\end{equation}

The Lagrangian deep learning model is trained by alternately carrying out gradient descent for each value of $\lambda$, and updating the $\lambda_i$ based on the resulting magnitudes of  constraint violation in its predicted solutions.

\section{Job Shop Scheduling}
\label{app:jss}

The Job Shop Scheduling (JSS) problem can be viewed as an integer optimization program with linear objective function and linear, disjunctive constraints. For JSS problems with $J$ jobs and $T$ machines, a particular instance is fully determined by the processing times $d_t^j$, along with machine assignments $\sigma_t^j$, and its solution consists of the resulting optimal task start times $s_j^t$.    The full problem specification is shown below in the system (\ref{model:jss}). The constraints (\ref{con:jss-b}) enforce precedence between tasks that must be scheduled in the specified order within their respective job. Constraints (\ref{con:jss-c}) ensure that no two tasks overlap in time when assigned to the same machine.

\subsection{Problem specification}
\label{app:jss_spec}

\begin{subequations}
    {
    \label{model:jss}
    \vspace{-6pt}
    \begin{flalign}
        \mathcal{P}(\bm{d}) = 
        \argmin_{\bm{s}}\;\;& u
        \label{jss-obj} \\
        \mbox{subject to:}\;\;&
                            u \geq s_T^j             && \!\!\!\!\! \forall j \!\in\! [J]        
                                \label{con:jss-a} \\
                            & s_{t+1}^j \geq s_t^j + {d}^{j}_{t}   
                                                                && \!\!\!\!\! \forall j \!\in\! [J-1], \forall t \!\in\! [T]    
                                \label{con:jss-b} \\
                            & s_t^{j} \geq s_{t'}^{j'} + {d}_{t'}^{j'}\; \lor\; s_{t'}^{j'} \geq s_{t}^{j} + {d}_{t}^{j} 
                                                                 && \!\!\!\!\! \forall j, j' \!\in\! [J], t,t' \!\in\! [T]
                                                                 \,\text{with}\, {\sigma}_{t}^{j} = {\sigma}_{t'}^{j'}      \label{con:jss-c} \\
                            & s_{t}^{j} \!\in\! \mathbb{N}           
                                                                && \!\!\!\!\! \forall j \!\in\! [J], t \!\in\! [T] \label{con:jss-d}
    \end{flalign}
    }
    \vspace{-12pt}
\end{subequations}

Given a predicted, possibly infeasible schedule $\hat{s}$, the degree of violation in each constraint must be measured in order to update the multipliers of the Lagrangian loss function. The violation of task-precedence constraints (\ref{con:jss-b})  and no-overlap constraint (\ref{con:jss-c}) are calculated as in (\ref{viol:tp}) and (\ref{viol:nv}), respectively. Note that the violation of the disjunctive no-overlap condition between two tasks is measured as the amount of time at which both tasks are scheduled simultaneously on some machine. 

\begin{subequations}
\begin{flalign}
    \label{viol:tp}
     &\nu_{10b}\hugP{\hat{s}_t^j, d_t^j} = 
     \max \hugP{0, \hat{s}_{t}^j + d_{t}^j - \hat{s}_{t+1}^j} \\
     \label{viol:nv}
     &\nu_{10c}\hugP{\hat{s}_t^j, d_t^j, \hat{s}_{t'}^{j'}, d_{t'}^{j'}} =
        \min\hugP{\nu_{10c}^L\hugP{\hat{s}_t^j, d_t^j, \hat{s}_{t'}^{j'}, d_{t'}^{j'}}, 
         \nu_{10c}^R\hugP{\hat{s}_t^j, d_t^j, \hat{s}_{t'}^{j'}, d_{t'}^{j'}}},
\end{flalign}
\end{subequations}
where
\begin{align*}
    \nu_{10c}^L\hugP{\hat{s}_t^j, d_t^j, \hat{s}_{t'}^{j'}, d_{t'}^{j'}} &= 
    \max\hugP{0, \hat{s}_t^j + d^t_j - \hat{s}_{t'}^{j'}}\\
        \nu_{10c}^R\hugP{\hat{s}_t^j, d_t^j, \hat{s}_{t'}^{j'}, d_{t'}^{j'}} &= 
    \max\hugP{0, \hat{s}_{t'}^{j'} + d^{t'}_{j'} - \hat{s}_{t}^{j}}.
\end{align*}

The Lagrangian-based deep learning model does not necessarily produce feasible schedules directly. An additional operation is required for the construction of feasible solutions, given the direct neural network outputs representing schedules. The model presented below is used to construct solutions that are integral, and feasible to the original problem constraints. Integrality follows from the total unimodularity of constraints  (\ref{recon:11a}, \ref{recon:11b}), which converts the no-overlap condition of the problem  (\ref{model:jss}) into addition task-precedence constraints following the order of predicted start times $\hat{s}$, denoted $\preceq_{\hat{\bm{s}}}$. By minimizing the makespan as in (\ref{model:jss}), this procedure ensures optimality of the resulting schedules subject to the imposed ordering. 

\begin{subequations}
\label{model:jss_rec}
    \begin{flalign}
    \Pi(\bm{s}) = \qquad
    {\argmin}_{\bm{s}}\;\;&   u     && \notag \\
        \mbox{subject to:}\;\;  
            &\eqref{con:jss-a}, \eqref{con:jss-b}&& \notag\\
            & s_t^{j} \geq s_{t'}^{j'} + {d}_{t'}^{j'} \!\!\!\!\!\!\!
            && \forall j, j' \!\in\! [J], t,t' \!\in\! [T] \; \text{s.t.} \; (j,t) \preceq_{\hat{\bm{s}}} (j', t') 
            \label{recon:11a} \\
            & s_{t}^{j} \geq 0           && \forall j \!\in\! [J], t \!\in\! [T] \label{recon:11b}
	\end{flalign}
\end{subequations}

\subsection{Dataset Details}
\label{app:jss_datasets}

The experimental setting, as defined by the training and test data, simulates a situation in which some component of a manufacturing system 'slows down', causing processing times to extend on all tasks assigned to a particular machine. Each experimental dataset is generated beginning with a root problem instance taken from the JSPLIB benchmark library for JSS instances. Further instances are generated by increasing processing times on one machine, uniformly over $5000$ new instances, to a maximum of $50$ percent increase over the initial values. To accommodate these incremental perturbations in problem data while keeping all values integral, a large multiplicative scaling factor is applied to all processing times of the root instance.
Targets for the supervised learning are generated by solving the individual instances according to the methodology proposed in Section \ref{sec:optimal_dataset}. A baseline set of solutions is generated for comparison, by solving individual instances in parallel with a time limit per instance of $1800$ seconds. 

The results presented in Section \ref{experimental_section} are taken from the best-performing models, with respect to optimality of the predicted solutions following application of the model (\ref{model:jss_rec}), among the results of a hyperparameter search. The model training follows the selection of parameters presented in Table \ref{tab:JSS_params}.

\begin{table}[!h]
\centering
\begin{tabular}{l@{\hspace{6pt}} |c@{\hspace{6pt}}
                 || @{\hspace{6pt}}l@{\hspace{6pt}} |c@{\hspace{6pt}}
  } 
	\toprule
      Parameter & Value  &	  Parameter & Value  \\
  	\midrule
  	\textbf{Epochs} & 500 & \textbf{Batch Size}& 16   \\
	\textbf{Learning rate} & $[1.25e^{-4} ,2e^{-3}]$ & \textbf{Batch Normalization}& False   \\
	\textbf{Dual learning rate} & $[1e^{-3} , 5e^{-2}]$   & \textbf{Gradient Clipping}& False    \\
	\textbf{Hidden layers} & 2  & \textbf{Activation Function}& ReLU    \\
	\bottomrule
\end{tabular}
\caption{JSS: Training Parameters}
\label{tab:JSS_params} 
\end{table}

\subsection{Network Architecture}
\label{app:jss_net}

The neural network architecture used to learn solutions to the JSS
problem takes into account the structure of its constraints,
organizing input data by individual job, and machine of the
associated tasks. When $\cI^{
(j)}_k$ and $\cI^{(m)}_k$ represent the input array indices corresponding
to job $k$ and machine $k$, the associated subarrays $d[\cI^{
(j)}_k]$ and  $d[\cI^{(m)}_k]$ are each passed from
the input array to a series of respective \emph{Job} and \emph{Machine layers}. The resulting arrays, one for every job and machine, are concatenated to form a single array and passed to further \emph{Shared Layers}. Each shared layer has size $2JT$ in the case of $J$ jobs and $T$ machines, and a final layer maps the output to an array of size $JM$, equal to the total number of tasks. This
architecture improves accuracy significantly in practice, when
compared with fully connected networks of comparable size.

\section{AC Optimal Power Flow}
\label{app:opf}

\subsection{Problem specification}
\label{app:opt_spec}

\begin{model}[t]
  {\small
  \caption{${\cal O}_{\text{OPF}}$: AC Optimal Power Flow}
  \label{model:ac_opf}
  \vspace{-6pt}
  \begin{align}
    \mbox{\bf variables:} \;\;
    & S^g_i, V_i \;\; \forall i\in N, \;\;
      S^f_{ij}   \;\; \forall(i,j)\in E \cup E^R \nonumber \\
    \mbox{\bf minimize:} \;\;
    & {\cO}(\bm{S^d}) = \sum_{i \in N} {c}_{2i} (\Re(S^g_i))^2 + {c}_{1i}\Re(S^g_i) + {c}_{0i} \label{ac_obj} \\
    \mbox{\bf subject to:} \;\; 
    & \angle V_{i} = 0, \;\; i \in N \label{eq:ac_0} \\
    & {v}^l_i \leq |V_i| \leq {v}^u_i     \;\; \forall i \in N \label{eq:ac_1} \\
    & {\theta}^{l}_{ij} \leq \angle (V_i V^*_j) \leq {\theta}^{u}_{ij} \;\; \forall (i,j) \in E  \label{eq:ac_2}  \\
    & {S}^{gl}_i \leq S^g_i \leq {S}^{gu}_i \;\; \forall i \in N \label{eq:ac_3}  \\
    & |S^f_{ij}| \leq {s}^{fu}_{ij}          \;\; \forall (i,j) \in E \cup E^R \label{eq:ac_4}  \\
    & S^g_i - {S}^d_i = \textstyle\sum_{(i,j)\in E \cup E^R} S^f_{ij} \;\; \forall i\in N \label{eq:ac_5}  \\ 
    & S^f_{ij} = {Y}^*_{ij} |V_i|^2 - {Y}^*_{ij} V_i V^*_j       \;\; \forall (i,j)\in E \cup E^R \label{eq:ac_6}
  \end{align}
  }
  \vspace{-12pt}
\end{model}

\emph{Optimal Power Flow (OPF)} is the problem of finding the best generator dispatch to meet the demands
in a power network, while satisfying challenging transmission constraints 
such as the nonlinear nonconvex AC power flow equations and also operational limits such as voltage and generation bounds.
Finding good OPF predictions are important, as a 5\% reduction in generation costs could save billions of dollars (USD) per year~\cite{Cain12historyof}.
In addition, the OPF problem is a fundamental building bock of many
applications, including security-constrained OPFs \cite{monticelli:87}), optimal transmission switching \cite{OTS}, capacitor placement \cite{baran:89}, and expansion planning \cite{verma:16}.

Typically, generation schedules are updated in intervals of 5
minutes \cite{Tong:11}, possibly using a solution to the OPF solved in
the previous step as a starting point. In recent years, the
integration of renewable energy in sub-transmission and distribution
systems has introduced significant stochasticity in front and behind
the meter, making load profiles much harder to predict and introducing
significant variations in load and generation. This uncertainty forces
system operators to adjust the generators setpoints with increasing
frequency in order to serve the power demand while ensuring stable
network operations. However, the resolution frequency to solve OPFs is
limited by their computational complexity. To address this issue,
system operators typically solve OPF approximations such as the linear
DC model (DC-OPF).  While these approximations are more efficient
computationally, their solution may be sub-optimal and induce
substantial economical losses, or they may fail to satisfy the
physical and engineering constraints.

Similar issues also arise in expansion planning and other
configuration problems, where plans are evaluated by solving a massive
number of multi-year Monte-Carlo simulations at 15-minute intervals
\cite{pachenew,Highway50}. Additionally, the stochasticity introduced
by renewable energy sources further increases the number of scenarios
to consider.  Therefore, modern approaches recur to the linear DC-OPF
approximation and focus only on the scenarios considered most
pertinent \cite{pachenew} at the expense of the fidelity of the
simulations.

A power network $\bm{\cN}$ can be represented as
a graph $(N, E)$, where the nodes in $N$ represent buses and the edges
in $E$ represent lines. The edges in $E$ are directed and $E^R$ is
used to denote those arcs in $E$ but in reverse direction.  The AC
power flow equations are based on complex quantities for current $I$,
voltage $V$, admittance $Y$, and power $S$, and these equations are a
core building block in many power system applications.
Model~\ref{model:ac_opf} shows the AC OPF formulation, with
variables/quantities shown in the complex domain.  Superscripts $u$ and
$l$ are used to indicate upper and lower bounds for variables. The
objective function ${\cO}(\bm{S^g})$ captures the cost of the
generator dispatch, with $\bm{S^g}$ denoting the vector of generator
dispatch values $(S^g_i \:|\: i \in N)$.  Constraint \eqref{eq:ac_0}
sets the reference angle to zero for the slack bus $i \in N$ to
eliminate numerical symmetries.  Constraints \eqref{eq:ac_1} and
\eqref{eq:ac_2} capture the voltage and phase angle difference bounds.
Constraints \eqref{eq:ac_3} and \eqref{eq:ac_4} enforce the generator
output and line flow limits.  Finally, Constraints \eqref{eq:ac_5}
capture Kirchhoff's Current Law and Constraints \eqref{eq:ac_6}
capture Ohm's Law.

The Lagrangian-based deep learning model is based on the model 
reported in \cite{Fioretto:AAAI-20}.

\subsection{Dataset Details}
\label{app:opf_datasets}

\begin{table}[t]
\centering
{
  \begin{tabular}{l@{\hspace{6pt}} |c@{\hspace{6pt}} 
                 c@{\hspace{6pt}} c@{\hspace{6pt}} c@{\hspace{6pt}}
                  c@{\hspace{6pt}} c@{\hspace{6pt}}}
	\toprule
	  Instance & \multicolumn{2}{c}{Size} 
    	       & \multicolumn{2}{c}{Total Variation}\\
  			   &  $|{N}|$ & $|{E}|$ & Standard Data & OD Data \\
  	\midrule
	\textbf{30\_ieee    } & 30 & 82&   $2.56570$ & $\bm{0.00118}$ \\
	\textbf{57\_ieee    } & 57&  160&  $11.5160$ & $\bm{0.00509}$ \\
	\textbf{89\_pegase  } & 89&  420&  $20.9309$ & $\bm{0.02538}$ \\
	\textbf{118\_ieee    }& 118& 372&  $40.2253$ & $\bm{0.01102}$ \\
	\textbf{300\_ieee    }& 300& 822&  $213.075$ & $\bm{0.13527}$ \\
	\bottomrule
	\end{tabular}
}
\caption{Standard vs OD training data: Total Variation.}
\label{tab:OPF_tv} 
\end{table}
Table~\ref{tab:OPF_tv} describes the power network benchmarks used,
including the number of buses $|{\cal N}|$, and transmission lines/transformers $|{\cal E}|$.
Additionally it presents a comparison of the total variation resulting 
from the two datasets. Note that the OD datasets have total variation which 
is orders of magnitude lower than their Standard counterparts.

\subsection{Network Architecture}
\label{app:opf_net}
The neural network architecture used to learn solutions to the OPF problem 
is a fully connected ReLU network composed of an input layer of size 
proportional to the number of loads 
in the power network. The architecture has 5 hidden layers, each of
size double  the number of loads in the power network, and a final layer
of size proportional to the number of generators in the network. 
The details of the learning models are reported in Table \ref{tab:OPF_params}.

\begin{table}[!h]
\centering
\begin{tabular}{l@{\hspace{6pt}} |c@{\hspace{6pt}}
                 || @{\hspace{6pt}}l@{\hspace{6pt}} |c@{\hspace{6pt}}
  } 
    \toprule
      Parameter & Value  &    Parameter & Value  \\
    \midrule
  	\textbf{Epochs}                & 20000  & \textbf{Batch Size}            &16\\
	\textbf{Learning rate}         & $[1e^{-5}, 1e^{-4}]$ & \textbf{Batch Normalization}   &True\\ 
	\textbf{Dual learning rate}    &$1e^{-4}$ & \textbf{Gradient Clipping}     &True\\ 
	\textbf{Hidden layers}         &5 & \textbf{Activation Function}   &LeakyReLU\\  
	\bottomrule
	\end{tabular}
\caption{OPF: Training Parameters}
\label{tab:OPF_params} 
\end{table}

\section{Additional Results}
\label{app:results}

\begin{table}[t]
\centering
{
\begin{tabular}{rl ll ll ll ll}
\toprule
Instance & \multicolumn{1}{c}{Size} 
         & \multicolumn{2}{c}{Prediction Error} 
         & \multicolumn{2}{c}{Constraint~Violation} 
         & \multicolumn{2}{c}{Optimality Gap (\%)}\\
\cmidrule(r){3-4} \cmidrule(r){5-6} \cmidrule(r){7-8}
  & No.~buses
  & Standard & OD 
  & Standard & OD 
  & Standard & OD \\
\midrule 
 IEEE-30   & 30  & $22.31$ & $\bm{0.11}$ & $0.063$& $\bm{0.00004}$ &  $6.28$ & $\bm{0.76}$ \\
 IEEE-57   & 57  & $83.61$ & $\bm{0.58}$ & $0.139$& $\bm{0.0002}$ &  $1.04$ & $\bm{0.66}$ \\
 Pegase-89 & 89  & $89.17$ & $\bm{2.78}$ & $1.353$ & $\bm{0.003}$ & $20.1$ & $\bm{0.83}$ \\
 IEEE-118  & 118 & $36.55$ & $\bm{0.54}$ & $1.330$ & $\bm{0.002}$ & $3.80$ & $\bm{0.36}$ \\
 IEEE-300  & 300 & $157.3$ & $\bm{2.27}$ & $1.891$ & $\bm{0.009}$ & $22.9$ & $\bm{0.12}$ \\
\bottomrule
\end{tabular}
}
  \caption{OPF -- Standard vs OD training data: prediction errors, constraint violations, and 
  optimality gap.}
    \label{tab:OPF_acc2}
\end{table}

Table \ref{tab:OPF_acc2} compares prediction errors and constraint violations
for the OD and Standard approach to data generation for the Optimal Power 
Flow problems. As expressed in the main paper, the results show that the models 
trained on the OD datset present 
predictions that are closer to their optimal target solutions (error expressed 
in MegaWatt (MW)), reduce the constraint violations (expressed as $L_1$-distance 
between the predictions and their projections), and improve the optimality gap, 
which is the relative difference in objectives between the predicted (feasible) 
solutions and the target ones.

\end{document}